\documentclass{article}

\usepackage{amsmath}
\usepackage{amsfonts}
\usepackage{graphicx, subfigure}
\usepackage[pdftex]{hyperref}
\usepackage[dutch,english]{babel}
\usepackage{amsmath,amssymb}
\usepackage{amsthm}
\usepackage[usenames,dvipsnames]{color}
\usepackage{tikz}
\usepackage{accents}
\usepackage{xcolor}
\usepackage{hyperref}
\usepackage{mathtools}
\usepackage{float}
\usepackage{multicol}
\usepackage{comment}
\usepackage[T1]{fontenc}
\usepackage{ stmaryrd }
\usepackage[font={small,it}]{caption}
\usetikzlibrary{decorations.pathmorphing}
\usepackage{tabu}
\usepackage{bm}

\usepackage[a4paper, total={6in, 8in}]{geometry}

\DeclareMathOperator{\mat}{Mat}

\DeclareMathOperator{\Id}{Id}

\definecolor{bloesh}{RGB}{160, 194, 185}
\definecolor{roo}{RGB}{241,23,70}

\newfont{\gothic}{eufm10}

\theoremstyle{plain}
\newtheorem{thr}{Theorem}[section]

\newtheorem{lem}[thr]{Lemma}

\newtheorem{prop}[thr]{Proposition}
\theoremstyle{definition}
\newtheorem{defi}[thr]{Definition}

\theoremstyle{remark}
\newtheorem{remk}[thr]{Remark}
\theoremstyle{remark}

\newcommand{\field}[1]{\mathbb{#1}}
\newcommand{\R}{\field{R}}
\newcommand{\N}{\field{N}}
\newcommand{\C}{\field{C}}
\newcommand{\Z}{\field{Z}}

\definecolor{bloesh}{RGB}{160, 194, 185}
\definecolor{roo}{RGB}{241,23,70}
\title{A new algorithm for computing idempotents of $\mathcal{R}$-trivial monoids}
\author{Eddie Nijholt\footnote{\mbox{Department of Mathematics, University of Illinois at Urbana-Champaign, USA, \href{mailto:eddie.nijholt@gmail.com}{eddie.nijholt@gmail.com} }  }, Bob Rink \footnote{\mbox{Department of Mathematics, Vrije Universiteit Amsterdam, the Netherlands, \href{mailto:b.w.rink@vu.nl}{b.w.rink@vu.nl} }}, S\"oren Schwenker  \footnote{\mbox{Department of Mathematics, Universit\"at Hamburg, Germany, \href{mailto:soeren.schwenker@uni-hamburg.de}{soeren.schwenker@uni-hamburg.de} }}}
\date{\today}

\begin{document}

\maketitle

\begin{abstract} \label{abstract}
The authors of \cite{Rtriv} provide an algorithm for finding a complete system of primitive orthogonal idempotents for $\C\mathcal{M}$, where $\mathcal{M}$ is any finite $\mathcal{R}$-trivial monoid. Their method relies on a technical result stating that  $\mathcal{R}$-trivial monoids are equivalent to so-called weakly ordered monoids. We provide an alternative algorithm, based only on the simple observation that an $\mathcal{R}$-trivial monoid may be realised by upper triangular matrices.  This approach is inspired by results in the field of coupled cell network dynamical systems, where $\mathcal{L}$-trivial monoids (the opposite notion) correspond to so-called feed-forward networks. We also show that our algorithm only has to be performed for $\Z\mathcal{M}$, after which a complete system of primitive orthogonal idempotents may be obtained for $R\mathcal{M}$, where $R$ is any ring with a known complete system of primitive orthogonal idempotents. In particular, the algorithm works if $R$ is any field.
\end{abstract}
\section{Introduction}
A monoid $\mathcal{M}$ may be seen as a generalisation of a group, in the sense that $\mathcal{M}$ is a set with an associative multiplication and a unit $e$, but where the elements do not necessarily have inverses. If an element $\sigma \in \mathcal{M}$ does have an inverse (that is, an element $\tau \in \mathcal{M}$ such that $\sigma\tau = \tau\sigma = e$), then the set 
\begin{equation}
\sigma\mathcal{M}:= \{ \sigma \kappa \, \mid \, \kappa \in \mathcal{M}\}
\end{equation}
necessarily equals $\mathcal{M}$. It follows that when $\mathcal{M}$ is a group, the sets $\sigma\mathcal{M}$ are all the same. Conversely, if the sets $\sigma\mathcal{M}$ are the same for all $\sigma$ in a monoid $\mathcal{M}$, then they are necessarily all equal to $e\mathcal{M} = \mathcal{M}$. It can then easily be shown that $\mathcal{M}$ is a group, see for example \cite{leftinv}. Hence, groups can be uniquely characterised as monoids for which the sets $\sigma\mathcal{M}$ are the same for all $\sigma \in \mathcal{M}$. Using this characterisation, an $\mathcal{R}$-\emph{trivial} monoid may be seen as a monoid that is farthest removed from being a group. That is, a monoid is called $\mathcal{R}$-trivial if the sets $\sigma\mathcal{M}$ are different for all $\sigma \in \mathcal{M}$. Likewise, a monoid is called $\mathcal{L}$-\emph{trivial} if the sets $\mathcal{M}\sigma$ are all different. \\
\indent The authors of \cite{Rtriv} provide an algorithm for finding a complete system of primitive orthogonal idempotents for the algebra $\C\mathcal{M}$, where $\mathcal{M}$ is any finite $\mathcal{R}$-trivial monoid. Their method uses a  technical result, which relates $\mathcal{R}$-trivial monoids to so-called weakly ordered monoids. By reinterpreting $\mathcal{R}$-trivial monoids as networks, which may in turn be represented by adjacency matrices, we are able to provide an alternative algorithm. The advantage of this approach is that we will only need basic combinatorial properties of upper triangular matrices. Moreover, due to the elementary nature of our techniques, our algorithm may be used to provide a complete system of primitive orthogonal idempotents for $R\mathcal{M}$, with $R$ any ring for which a complete system of primitive orthogonal idempotents is known. In fact, such a system for $R\mathcal{M}$ may be obtained immediately by combining the one our algorithm produces for $\Z\mathcal{M}$ with one for $R$. In particular, performing our algorithm for  $\Z\mathcal{M}$ yields a complete system of primitive orthogonal idempotents for $k\mathcal{M}$, with $k$ any field. More details on the network setting that inspired our approach are given in Remark \ref{netwoork}. The observation that $\mathcal{R}$-trivial monoids may be represented by upper triangular matrices is not new, see for example \cite{monbook}, Corollary 10.9. However, we are not aware of anyone using this fact to decompose $R\mathcal{M}$ in this straightforward manner. \\
\indent The remainder of this paper is organised as follows. In Section \ref{Preliminaries} we introduce the basic notions that we will be using throughout the paper. In Section \ref{realisisa} we then give an explicit realisation of $R\mathcal{M}$ as upper triangular matrices with entries in the ring $R$. Next, we describe our algorithm for $\Z\mathcal{M}$ in Section \ref{zet}, and we show in Section \ref{gennn} how this result leads to a complete system of primitive orthogonal idempotents for $R\mathcal{M}$. Finally, we illustrate our algorithm with a small example in Section \ref{exammp}.


\section{Preliminaries}\label{Preliminaries}
Recall that a monoid $\mathcal{M}$ is a set, together with an associative multiplication $\mathcal{M} \times \mathcal{M} \rightarrow \mathcal{M}$ and a unit $e \in \mathcal{M}$ such that $e \sigma = \sigma  e = \sigma$ for all $\sigma \in \mathcal{M}$. A monoid is called $\mathcal{R}$-\emph{trivial} if the sets $\sigma\mathcal{M}:= \{\sigma \mu \, \mid \, \mu \in \mathcal{M}\}$ are different for all $\sigma \in \mathcal{M}$. That is, when $\sigma\mathcal{M} = \tau\mathcal{M}$ for $\sigma,\tau \in \mathcal{M}$ implies $\sigma = \tau$. 
\begin{defi}
Let $R$ denote any ring (commutative or otherwise) with $1 \not= 0$, and let $\mathcal{M}$ be a finite monoid. We define the $R$-algebra of formal sums
\begin{equation}
R\mathcal{M}= \left\{ \left. \sum_{\mu \in \mathcal{M}} r_{\mu}\cdot \mu \, \right|  r_{\mu} \in R \hspace{4pt} \forall \, \mu \in \mathcal{M} \right\} 
\end{equation}
with coordinate-wise addition, i.e.
\begin{equation}
\sum_{\mu \in \mathcal{M}} r_{\mu} \cdot \mu  + \sum_{\mu \in \mathcal{M}} s_{\mu} \cdot \mu = \sum_{\mu \in \mathcal{M}} (r_{\mu} + s_{\mu}) \cdot \mu \, ,
\end{equation}
and multiplication 
\begin{align}
\left(\sum_{\sigma \in \mathcal{M}} r_{\sigma} \cdot \sigma \right)\cdot \left( \sum_{\tau \in \mathcal{M}} s_{\tau} \cdot \tau \right) &=  \sum_{\sigma \in \mathcal{M}} \sum_{\tau \in \mathcal{M}} (r_{\sigma} s_{\tau})\cdot [\sigma  \tau] = \sum_{\mu \in \mathcal{M}} \left( \sum_{\substack{\sigma, \tau \in \mathcal{M} \\ \sigma  \tau = \mu}} r_{\sigma} s_{\tau}  \right) \cdot \mu\, .
\end{align} 
The ring $R$ furthermore acts on $R\mathcal{M}$ from the left by identifying $R$ with $Re \subset R\mathcal{M}$. That is, we have
\begin{equation}
r \cdot \left(\sum_{\sigma \in \mathcal{M}} r_{\sigma} \cdot \sigma \right) = (r\cdot e) \cdot \left(\sum_{\sigma \in \mathcal{M}} r_{\sigma} \cdot \sigma \right) = \sum_{\sigma \in \mathcal{M}} (rr_{\sigma})\cdot \sigma\, ,
\end{equation}
for $r \in R$.  Similarly, $R$ acts on $R\mathcal{M}$ from the right. Finally, if we are given $\sigma \in \mathcal{M}$, we will often simply write $\sigma$ and $-\sigma$ for $1\cdot \sigma \in R\mathcal{M}$ and $-1\cdot \sigma \in R\mathcal{M}$, respectively. $\hfill \triangle$
\end{defi}
\noindent Just as in \cite{Rtriv}, our goal is to find a complete system of primitive orthogonal idempotents for $R\mathcal{M}$ when $\mathcal{M}$ is $\mathcal{R}$-trivial. That is, we are looking for a finite set of non-zero elements $E_1, \dots, E_p \subset R\mathcal{M}$ such that 
\begin{enumerate}\label{cspoi}
\item Each $E_i$ is idempotent. That is, $E_i ^2 = E_i$ for all $i \in \{1, \dots, p\}$.
\item The $E_i$ are orthogonal, meaning that $E_i \cdot E_j = 0$ for all $i,j \in \{1, \dots, p\}$ with $i \not=j$.
\item The $E_i$ are primitive. In other words, if $E_i = X + Y$ for some $X,Y \in R\mathcal{M}$ with $X^2 = X$, $Y^2 = Y$ and $XY = YX = 0$, then either $X=0$ or $Y=0$.
\item It holds that $\sum_{i=1}^p E_i = e$ (completeness).
\end{enumerate}
We will furthermore assume that we are given a complete system of primitive orthogonal idempotents $\epsilon_1, \dots, \epsilon_q$ for $R$ (that is, $\epsilon_1$ till $\epsilon_q$ are non-zero elements of $R$ satisfying the four conditions above, with $X,Y \in R$ in point 3 and $\epsilon_1 + \dots + \epsilon_q = 1$ in point 4). Note that for any field, and more generally for any domain, a complete system of primitive orthogonal idempotents is given by only the element $1$. This is because the only idempotents in a domain are $0$ and $1$. 

\section{Realisation as upper triangular matrices}\label{realisisa}
From here on out, $\mathcal{M}$ will always denote a finite $\mathcal{R}$-trivial monoid. We will analyse $R\mathcal{M}$ by showing that it is isomorphic to an $R$-algebra of upper triangular matrices. More precisely, we have the following definition.
\begin{defi}\label{theuuu}
Let $R$ be a ring with $1 \not= 0$. We denote by $\mat(R,n)$ the $R$-algebra of all $n \times n$ matrices with entries in $R$. More precisely,  multiplication in $\mat(R,n)$ is given by
\begin{equation}
(A\cdot B)_{i,j} = \sum_{k = 1}^n A_{i,k}B_{k,j} \, ,
\end{equation}
for $A,B \in \mat(R,n)$ and $i,j \in \{1, \dots, n\}$. Addition is defined entry-wise, and $R$ acts on $\mat(R,n)$ by $r \cdot A \cdot s = (r\Id) A (s\Id)$ for $r,s \in R$. Here we use the matrix $r\Id$ given by $(r\Id)_{i,j} = r\delta_{i,j}$, with $\delta_{i,j}$ the Kronecker delta function. Given $r \in R$ and $A \in \mat(R,n)$ we will simply write $rA$ for $r \cdot A$. $\hfill \triangle$
\end{defi}
\noindent Next, we set $n:= \#\mathcal{M}$ and we choose an ordering of the elements in $\mathcal{M}$, say $\mathcal{M} = \{\sigma_1, \dots, \sigma_n \}$. Given $\sigma_i, \sigma_j \in \mathcal{M}$, we may write $A_{\sigma_i, \sigma_j} := A_{i,j}$ for $A \in \mat(R,n)$. In particular, this means we may write $A_{\sigma, \tau}$ for any two elements $\sigma, \tau \in \mathcal{M}$. Using this notation, we define elements $U_{\sigma} \in \mat(R,n)$ for $\sigma \in \mathcal{M}$ by
\begin{equation}
(U_{\sigma})_{\tau, \kappa} = \delta_{\tau  \sigma, \kappa } = \begin{cases} 1 &\mbox{if }\tau  \sigma = \kappa \\0 &\mbox{otherwise,}  \end{cases}\, 
\end{equation}
for all $\tau, \kappa \in \mathcal{M}$. \\

\noindent The following lemma, together with Lemma \ref{key2}, may be considered the key observation of this paper. It is motivated by a construction from the field of network dynamical systems, called the \emph{fundamental network}, see for example \cite{fibr}.
\begin{lem} \label{key1}
The map 
\begin{align}
\Psi: \,&R\mathcal{M} \rightarrow \mat(R,n)\\ \nonumber
&\sum_{\mu \in \mathcal{M}} r_{\mu} \cdot \mu \mapsto \sum_{\mu \in \mathcal{M}} r_{\mu}U_{\mu} 
\end{align}
is an injective morphism of unitary $R$-algebras.
\end{lem}
\begin{proof}
First, we claim that $U_{\sigma}U_{\tau} = U_{\sigma  \tau}$ for all $\sigma, \tau \in \mathcal{M}$. Indeed, a calculation shows that
\begin{align}
(U_{\sigma}U_{\tau})_{\mu, \kappa} &= \sum_{\rho \in \mathcal{M}} (U_{\sigma})_{\mu, \rho}(U_{\tau})_{\rho, \kappa} = \sum_{\rho \in \mathcal{M}} \delta_{\mu  \sigma, \rho} \delta_{\rho  \tau,  \kappa} =  \delta_{\mu  \sigma  \tau, \kappa} = (U_{\sigma   \tau})_{\mu, \kappa} \, ,
\end{align}
for all $\mu, \kappa\in \mathcal{M}$. Next, we note that each of the matrices $U_{\sigma}$ for $\sigma \in \mathcal{M}$ has only entries given by $0$ and $1$. As a result, $U_{\sigma}$ commutes with the diagonal matrix $r\Id$ for all $r \in R$, and we conclude that $\Psi$ respects the left and right action of $R$. It moreover follows that 
\begin{equation}
\Psi(r\cdot \sigma)\Psi( s\cdot \tau)  = rU_{\sigma}sU_{\tau} = rsU_{\sigma}U_{\tau} = rsU_{\sigma\tau}  = \Psi(rs\cdot \sigma\tau) = \Psi((r\cdot \sigma)( s\cdot \tau) )\, ,
\end{equation}
for all $r,s \in R$ and $\sigma, \tau \in \mathcal{M}$. From the fact that $\Psi(x) + \Psi(y) = \Psi(x+y)$ for all $x,y \in R\mathcal{M}$, we conclude that in fact $\Psi(x)\Psi(y) = \Psi(xy)$ for general $x,y \in R\mathcal{M}$.  Note that we furthermore have $U_e = \Id$. Finally, it remains to show that $\Psi$ is injective. To this end, we see that
\begin{equation}
\left(\sum_{\mu \in \mathcal{M}} r_{\mu}U_{\mu}\right)_{\hspace{-0.15cm}e, \kappa} \hspace{-0.1cm}= \sum_{\mu \in \mathcal{M}} (r_{\mu}U_{\mu})_{e, \kappa} = \sum_{\mu \in \mathcal{M}} r_{\mu}\delta_{\mu,\kappa} = r_{\kappa}\, ,
\end{equation}
for all $\kappa \in \mathcal{M}$. It follows that $\Psi(x) = 0$ for $x \in R\mathcal{M}$ necessarily implies $x=0$. This proves the lemma.
\end{proof}
\begin{lem}\label{key2}
There exists an ordering of the elements of $\mathcal{M}$ for which every matrix in the image of $\Psi$ is in upper triangular form.
\end{lem}
\begin{proof}
We pick any ordering of $\mathcal{M}$ such that $\#\sigma_1\mathcal{M} \geq \#\sigma_2\mathcal{M} \geq \dots \geq \#\sigma_n\mathcal{M}$. It suffices to show that every matrix $U_{\sigma}$ is upper triangular, and so that $(U_{\sigma})_{\sigma_i, \sigma_j} = 0$ whenever $i > j$. Suppose otherwise, so that $(U_{\sigma})_{\sigma_i, \sigma_j} = 1$ for some $i > j$. By definition of $U_{\sigma}$, we have $\sigma_i  \sigma = \sigma_j$. Hence, we see that $\sigma_j\mathcal{M} = \sigma_i  \sigma \mathcal{M} \subset \sigma_i\mathcal{M}$. As we have $i > j$, it holds that $\#\sigma_j\mathcal{M} \geq \#\sigma_i\mathcal{M}$, and we conclude that $\sigma_i\mathcal{M} = \sigma_j\mathcal{M}$. This of course directly contradicts the definition of an $\mathcal{R}$-trivial monoid, and so we find that every element in the image of $\Psi$ is indeed in upper triangular form. This proves the lemma.
\end{proof}
\noindent We will assume from now on that the elements of $\mathcal{M}$ are ordered as in Lemma \ref{key2}, so that every matrix in the image of $\Psi$ is in upper triangular form. We will denote this image of $R\mathcal{M}$ under $\Psi$ by $\mathcal{U}^R_{\mathcal{M}} \subset \mat(R,n)$. As expected, the diagonal entries of elements in $\mathcal{U}^R_{\mathcal{M}}$ will play a crucial role in finding idempotents in $R\mathcal{M}$. They will motivate the following definition.

\begin{defi}
Given an element $\sigma \in \mathcal{M}$, we define the set 
\begin{equation}
\mathcal{L}_{\sigma} := \{\tau \in \mathcal{M} \, \mid \, \sigma  \tau = \sigma\}\, .
\end{equation}
Note that $e \in \mathcal{L}_{\sigma}$ for all $\sigma \in \mathcal{M}$, so that none of the sets $\mathcal{L}_{\sigma}$ is empty. It may however happen that two of the sets $\mathcal{L}_{\sigma}$ coincide. We will describe when this happens by an equivalence relation $\sim$ on $\mathcal{M}$. In other words, we have $\sigma \sim \kappa \iff \mathcal{L}_{\sigma} = \mathcal{L}_{\kappa}$ for $\sigma, \kappa \in \mathcal{M}$. We denote the class of an element $\sigma \in \mathcal{M}$ under $\sim$ by $[\sigma]$, which we call the \emph{loop-type} of $\sigma$. $\hfill \triangle$
\end{defi}

\begin{remk}\label{netwoork}
The term `loop-type' again comes from an interpretation as coupled cell networks. Specifically, \emph{feed-forward networks} are networks where there is a partial ordering on the nodes, such that every node only `feels' itself and those nodes that are higher in the ordering. Another way of saying this is that the only loops in the network are self-loops. Such networks are known to display unusual switching or amplifying dynamical behaviour, see for example \cite{claire} and \cite{RinkSanders2}. It can furthermore be shown that (under some technical conditions) a network is a feed-forward network, if and only if it is a quotient of the left Cayley graph of an $\mathcal{L}$-trivial monoid (the opposite of an $\mathcal{R}$-trivial monoid). In this more visual approach to monoids, two elements of $\mathcal{M}$ have the same loop-type if and only if their corresponding nodes in the Cayley graph have the same self-loops. \\
\indent More generally, large classes of networks may be realised as quotients of the left Cayley graph of a finite monoid, or of graphs encoding more general algebraic structure. This observation enables one to translate network structure on a dynamical system to so-called \emph{hidden symmetry}. In turn, this allows one to preserve network structure along many dynamical reduction techniques, and to offer explanations for the various invariant spaces, spectral degeneracies and unusual bifurcations often observed in network dynamical systems. See  \cite{fibr, RinkSanders2,  RinkSanders3, CCN,  cen, proj, Gensym, Schwenker} for more on this formalism. $\hfill \triangle$
\end{remk}

\begin{remk}\label{Uiffff}
Note that for any two elements $\kappa, \sigma \in \mathcal{M}$ we have $(U_{\kappa})_{\sigma, \sigma} = \delta_{\sigma\kappa, \sigma}$. This gives the useful observation that $(U_{\kappa})_{\sigma, \sigma} = 1$ if and only if we have $\kappa \in \mathcal{L}_{\sigma}$. $\hfill \triangle$
\end{remk}

\begin{lem}\label{diagisloop}
Let $U = \sum_{\mu \in \mathcal{M}} r_{\mu}U_{\mu}$ be an element of $\mathcal{U}^R_{\mathcal{M}}$ and let $\sigma \in \mathcal{M}$ be given. The $\sigma$-diagonal entry $U_{\sigma, \sigma}$ of $U$ is given by
\begin{equation}
U_{\sigma, \sigma} = \sum_{\mu \in \mathcal{L}_{\sigma}} r_{\mu}\, .
\end{equation}
In particular, two elements $\sigma, \tau \in \mathcal{M}$ have the same loop-type, if and only if for every matrix $U \in\mathcal{U}^R_{\mathcal{M}}$ the $\sigma$-diagonal entry of $U$ equals the $\tau$-diagonal entry of $U$. \end{lem}

\begin{proof}
The lemma follows directly from the definition of $U_{\sigma}$. Specifically, we find
\begin{equation}
U_{\sigma, \sigma} =  \sum_{\mu \in \mathcal{M}} r_{\mu}(U_{\mu})_{\sigma, \sigma} = \sum_{\mu \in \mathcal{M}} r_{\mu}\delta_{\sigma  \mu, \sigma} = \sum_{\mu \in \mathcal{L}_{\sigma}} r_{\mu}\, .
\end{equation}
It follows that $U_{\sigma, \sigma} = U_{\tau, \tau}$ whenever $[\sigma] = [\tau]$. Conversely, if $[\sigma] \not= [\tau]$ then we may choose $\kappa \in \mathcal{M}$ such that $\kappa \in \mathcal{L}_{\sigma}$ but $\kappa \notin \mathcal{L}_{\tau}$ (without loss of generality). By Remark \ref{Uiffff} we find $(U_{\kappa})_{\sigma, \sigma} = 1$, whereas $(U_{\kappa})_{\tau, \tau} = 0$. This finishes the proof.
\end{proof}
\noindent We end this section with a result that may greatly simplify determining loop-type equivalence. To this end, we need the following important lemma.

\begin{lem}\label{multi}
For  $\sigma, \tau_1, \tau_2 \in \mathcal{M}$ we have $\tau_1, \tau_2 \in \mathcal{L}_{\sigma}$ if and only if $\tau_1  \tau_2 \in \mathcal{L}_{\sigma}$.
\end{lem}

\begin{proof}
By Remark \ref{Uiffff}, we have $(U_{\tau_i})_{\sigma, \sigma} = 1$ if and only if $\tau_i \in \mathcal{L}_{\sigma}$, for $i=1,2$. Note also that any entry of $U_{\tau_i}$ is either $0$ or $1$. The statement of the lemma now follows immediately from the identity $(U_{\tau_1  \tau_2})_{\sigma, \sigma} = (U_{\tau_1})_{\sigma, \sigma} (U_{\tau_2})_{\sigma, \sigma} $, which holds because the matrices $U_{\mu}$ for $\mu \in \mathcal{M}$ are all upper triangular. \end{proof}
\noindent We now assume we are given a generating set $\mathcal{S}$ for $\mathcal{M}$. In what follows, we write
\begin{equation}
\mathcal{L}_{\sigma}^{\mathcal{S}} := \mathcal{L}_{\sigma} \cap \mathcal{S} =  \{\kappa \in \mathcal{S} \, \mid \, \sigma  \kappa = \sigma\}\, ,
\end{equation}
for an element $\sigma \in \mathcal{M}$.
\begin{prop}\label{loopiss}
Let $\mathcal{S}$ be a generating set for $\mathcal{M}$. Two elements $\sigma, \sigma' \in \mathcal{M}$ have the same loop-type if and only if the sets $\mathcal{L}_{\sigma}^{\mathcal{S}}$ and $\mathcal{L}_{\sigma'}^{\mathcal{S}}$ agree.
\end{prop}

\begin{proof}
If $\sigma$ and $\sigma'$ have the same loop-type, then $\mathcal{L}_{\sigma} = \mathcal{L}_{\sigma'}$, so that $\mathcal{L}_{\sigma} \cap \mathcal{S} = \mathcal{L}_{\sigma'} \cap \mathcal{S}$. \\
Conversely, suppose that $\mathcal{L}_{\sigma}^{\mathcal{S}} = \mathcal{L}_{\sigma'}^{\mathcal{S}}$. Let $\tau \in \mathcal{L}_{\sigma}$ be given, and write $\tau = \tau_1  \dots  \tau_k$ for some generators $\tau_1, \dots, \tau_k \in \mathcal{S}$. By repeated application of Lemma \ref{multi} we find $\tau_1, \dots, \tau_k \in \mathcal{L}_{\sigma}$, and so $\tau_1, \dots, \tau_k \in \mathcal{L}_{\sigma}^{\mathcal{S}}$. We conclude that $\tau_1, \dots, \tau_k \in \mathcal{L}_{\sigma'}^{\mathcal{S}} \subset  \mathcal{L}_{\sigma'}$. Again by repeated application of Lemma \ref{multi} we find $\tau = \tau_1  \dots  \tau_k \in  \mathcal{L}_{\sigma'}$, and we conclude that $\mathcal{L}_{\sigma} \subset \mathcal{L}_{\sigma'}$ By reversing the roles of $\sigma$ and $\sigma'$ we arrive at $\mathcal{L}_{\sigma} = \mathcal{L}_{\sigma'}$, so that indeed $[\sigma] = [\sigma']$. This completes the proof.
\end{proof}

\section{The case $R = \Z$}\label{zet}
In this section we present the algorithm for finding a complete system of primitive orthogonal idempotents in the case of $R = \Z$. We then show in the next section how to use this to find a complete system of primitive orthogonal idempotents for a general choice of $R$. As expected, an important role will be played by the diagonal of an element of  $\mathcal{U}^{\Z}_{\mathcal{M}}$. We start by fixing some notation, after which we list some easy observations.

\begin{defi}\label{defip}
Given an element $U \in \mathcal{U}^{\Z}_{\mathcal{M}}$, we define the set 
\begin{equation}
\mathcal{D}_U := \{\sigma \in \mathcal{M}\, \mid \, U_{\sigma, \sigma} = 1\}\, .
\end{equation}
We also write $d_U := \#\mathcal{D}_U$ for the number of ones on the diagonal of $U$, and we set $n = \#\mathcal{M}$ as before. Note that $d_U$ may be described independently of the identification between $\Z\mathcal{M}$ and $\mathcal{U}^{\Z}_{\mathcal{M}}$, for example by
\begin{equation}
d_U = \#\left\{\sigma \in \mathcal{M} \, \left| \, \sum_{\tau \in \mathcal{L}_{\sigma}} r_{\tau} =  1 \right. \right\}\, ,
\end{equation}
where
\begin{equation}
U = \sum_{\tau \in \mathcal{M}} r_{\tau} \cdot \tau\, .
\end{equation}
This follows directly from Lemma \ref{diagisloop}. Lastly, we define 
\begin{equation}
\mathcal{P}_{a,b}(U) := \Id-(\Id-U^{a})^b \in \mathcal{U}^{\Z}_{\mathcal{M}}\,
\end{equation}
for all $U \in \mathcal{U}^{\Z}_{\mathcal{M}}$ and $a,b \in \Z_{\geq 0}$. $\hfill \triangle$
\end{defi}

\begin{lem}\label{idemz}
\leavevmode
\begin{enumerate}
\item An idempotent element of  $\mathcal{U}^{\Z}_{\mathcal{M}}$ has only zeroes and ones on its diagonal.
\item If an idempotent element of  $\mathcal{U}^{\Z}_{\mathcal{M}}$ has only ones on its diagonal at those places that correspond to a single loop-type, then it is primitive. More precisely, if  						$E \in \mathcal{U}^{\Z}_{\mathcal{M}}$ is an idempotent such that the set $\mathcal{D}_E$ is a single equivalence class under $\sim$, then $E = X + Y$ with $X$ and $Y$ idempotent implies $X = 0$ or 		$Y = 0$.
\item Let $U \in \mathcal{U}^{\Z}_{\mathcal{M}}$ be an element with only zeroes and ones on its diagonal, and suppose we have $a,b \in \Z_{\geq 0}$ such that $a \geq n-d_U$ and $b \geq d_U$. Then $\mathcal{P}_{a,b}(U)$ is an idempotent element with the same diagonal as $U$.
\end{enumerate}
\end{lem}

\begin{proof}
The first part follows from the fact that diagonal elements of triangular matrices are multiplicative. That is, for all $\sigma \in \mathcal{M}$ and $U, U' \in \mathcal{U}^{R}_{\mathcal{M}}$ we have $(UU')_{\sigma, \sigma} = U_{\sigma, \sigma}U'_{\sigma, \sigma}$. It follows that any diagonal element of an idempotent in $\mathcal{U}^{R}_{\mathcal{M}}$ is an idempotent in $R$. In particular, for $R= \Z$ we find that the diagonal can only contain zeroes and ones.\\
\indent For the second part, suppose that an idempotent $U \in \mathcal{U}^{\Z}_{\mathcal{M}}$ can be written as the sum of two idempotents $X$ and $Y$. It follows from the first part that
\begin{equation}
\mathcal{D}_E = \mathcal{D}_X \sqcup  \mathcal{D}_Y\, .
\end{equation}
However, Lemma \ref{diagisloop} tells us that diagonal entries corresponding to loop-type equivalent monoid elements are always the same. Therefore, the assumption on $E$ means that either $X$ or $Y$ has only zeroes on its diagonal. Let us suppose this is the case for $X$. It follows that $X$ is a strictly upper triangular matrix, and is therefore nilpotent (i.e. $X^N = 0$ for some $N \in \N$). As $X$ is also assumed to be idempotent, we find $X = X^2 = \dots = X^N = 0$, and so $X = 0$. \\
\indent For the third part, we first show that $U$ and $\mathcal{P}_{a,b}(U)$ have the same diagonal when $a \geq n-d_U$ and $b \geq d_U$. First, it can be seen that $U$ and $U^a$ have the same diagonal. (Note that $a=0$ can only occur for $d_U = n$, in which case both $U$ and $U^0 = \Id$ have only ones on the diagonal). Next, we see that the diagonal of $(\Id-U^a)$ equals that of $U^a$ after changing the zeroes into ones and vice versa. Again, the diagonal of $(\Id-U^a)$ equals that of $(\Id-U^a)^b$. (Note that $b=0$ only occurs for $d_U = 0$, in which case both $(\Id-U^a)$ and $(\Id-U^a)^0 = \Id$ have only ones on the diagonal). Finally, we see that the diagonal of $\Id - (\Id-U^a)^b$ again equals that of $(\Id - U^a)^b$, but with the ones and zeroes reversed. It follows that the diagonal of $U$ indeed equals that of $\Id - (\Id - U^a)^b = \mathcal{P}_{a,b}(U)$. \\
It remains to show that $\mathcal{P}_{a,b}(U)$ is an idempotent. To this end, note that any element $U \in \mathcal{U}^{\Z}_{\mathcal{M}}$ with only zeroes and ones on the diagonal is conjugate to a real matrix of the form
\begin{equation}\label{matje1}
U \sim \left(
\begin{array}{ccc|ccc}
0 & & \bullet & & &\\
   & \ddots & &  & 0 & \\
 0 & & 0 & & &\\ \hline
 & & & 1 & & \bullet \\
& 0 &  &   & \ddots & \\
& & &   0 & & 1 \\
\end{array}
\right) \, .
\end{equation}
This follows for example by using Jordan normal form. Note also that the matrix of \eqref{matje1} has exactly $d_U$ ones on the diagonal, as this is the algebraic multiplicity of the eigenvalue $1$. It follows that $U^{a}$ is of the form
\begin{equation}\label{matje2}
U^{a} \sim \left(
\begin{array}{ccc|ccc}
0 & & 0 & & &\\
   & \ddots & &  & 0 & \\
 0 & & 0 & & &\\ \hline
 & & & 1 & & \bullet \\
& 0 &  &   & \ddots & \\
& & &   0 & & 1 \\
\end{array}
\right) \, .
\end{equation}
Likewise, we find 
\begin{align}\label{matje3}
\Id - U^{a} \sim &\left(
\begin{array}{ccc|ccc}
1 & & 0 & & &\\
   & \ddots & &  & 0 & \\
 0 & & 1 & & &\\ \hline
 & & & 0 & & \bullet \\
& 0 &  &   & \ddots & \\
& & &   0 & & 0 \\
\end{array}
\right) \, , \quad
(\Id - U^{a})^{b} \sim \left(
\begin{array}{ccc|ccc}
1 & & 0 & & &\\
   & \ddots & &  & 0 & \\
 0 & & 1 & & &\\ \hline
 & & & 0 & & 0 \\
& 0 &  &   & \ddots & \\
& & &   0 & & 0 \\
\end{array}
\right)  \, , \nonumber \\
\Id - (\Id - U^{a})^{b} \sim &\left(
\begin{array}{ccc|ccc}
0 & & 0 & & &\\
   & \ddots & &  & 0 & \\
 0 & & 0 & & &\\ \hline
 & & & 1 & & 0 \\
& 0 &  &   & \ddots & \\
& & &   0 & & 1 \\
\end{array}
\right) \, ,
\end{align}
so that $\Id - (\Id - U^{a})^{b} = \mathcal{P}_{a,b}(U)$ is indeed an idempotent matrix. This finishes the proof.
\end{proof}


\noindent Next, we want to construct a set of elements with ones on the diagonal for entries corresponding to exactly one loop-type, and zeroes for all other diagonal entries. Lemma \ref{idemz} (parts 2 and 3) then allows us to construct primitive idempotents.

\begin{prop}\label{diagfix}
Let $\sigma \in \mathcal{M}$ be given, and suppose $\mathcal{S}$ is a generating set for $\mathcal{M}$. As before, we set $\mathcal{L}_{\sigma}^{\mathcal{S}}:= \mathcal{L}_{\sigma} \cap \mathcal{S}$. The element 
\begin{equation}\label{erjee}
T_{[\sigma]} := \prod_{\tau \in \mathcal{L}_{\sigma}^{\mathcal{S}}} U_{\tau} \hspace{-0.2cm} \prod_{\kappa \in \mathcal{S} \setminus \mathcal{L}_{\sigma}^{\mathcal{S}}} (\Id - U_{\kappa})
\end{equation}
has diagonal entries given by
\begin{equation}
(T_{[\sigma]})_{\mu, \mu} =  \begin{cases} 1 &\mbox{if } \mu  \sim \sigma \\0 &\mbox{otherwise,}  \end{cases}
\end{equation}
for $\mu \in \mathcal{M}$. This holds independently of the order in which the product in \eqref{erjee} is evaluated.
\end{prop}

\begin{proof}
First, let $\mu$ be loop-type equivalent to $\sigma$. By Proposition \ref{loopiss} we have $\mathcal{L}_{\sigma}^{\mathcal{S}} = \mathcal{L}_{\mu}^{\mathcal{S}}$, so that we may write 
\begin{equation}\label{erjee2}
T_{[\sigma]} = \prod_{\tau \in \mathcal{L}_{\mu}^{\mathcal{S}}} U_{\tau} \hspace{-0.2cm} \prod_{\kappa \in \mathcal{S} \setminus \mathcal{L}_{\mu}^{\mathcal{S}}} (\Id - U_{\kappa})\, .
\end{equation}
In particular, we find
\begin{equation}\label{erjee3}
(T_{[\sigma]})_{\mu, \mu} = \prod_{\tau \in \mathcal{L}_{\mu}^{\mathcal{S}}} (U_{\tau})_{\mu, \mu}  \hspace{-0.2cm} \prod_{\kappa \in \mathcal{S} \setminus \mathcal{L}_{\mu}^{\mathcal{S}}} (\Id - U_{\kappa})_{\mu, \mu} \, .
\end{equation}
Now, for $\tau \in \mathcal{L}_{\mu}^{\mathcal{S}}$ we find $(U_{\tau})_{\mu, \mu} = 1$ by Remark \ref{Uiffff}. On the other hand,  $\kappa \in \mathcal{S} \setminus \mathcal{L}_{\mu}^{\mathcal{S}}$ implies $(U_{\kappa})_{\mu, \mu} = 0$, so that again $(\Id - U_{\kappa})_{\mu, \mu} = 1$. We conclude that indeed $(T_{[\sigma]})_{\mu, \mu} = 1$. \\
Next, suppose that $[\mu] \not= [\sigma]$, so that $\mathcal{L}_{\mu} \not= \mathcal{L}_{\sigma}$. By Proposition \ref{loopiss} we also have $\mathcal{L}_{\mu}^{\mathcal{S}} \not= \mathcal{L}_{\sigma}^{\mathcal{S}}$, so that there either exists a generator $\tau$ contained in $\mathcal{L}_{\sigma}^{\mathcal{S}}$, but not in $\mathcal{L}_{\mu}^{\mathcal{S}}$, or a generator $\kappa$ contained in $\mathcal{L}_{\mu}^{\mathcal{S}}$, but not in $\mathcal{L}_{\sigma}^{\mathcal{S}}$. In the first case, we find $(U_{\tau})_{\mu, \mu} = 0$, so that $(T_{[\sigma]})_{\mu, \mu} = 0$. In the second case, we get $(U_{\kappa})_{\mu, \mu} = 1$, and  so $(\Id - U_{\kappa})_{\mu, \mu} = 0$. As we have $\kappa \in \mathcal{S} \setminus \mathcal{L}_{\sigma}^{\mathcal{S}}$, we again conclude that $(T_{[\sigma]})_{\mu, \mu} = 0$.\\
Finally, we note that the discussion above is independent of the order in which the product of \eqref{erjee} is evaluated. More generally, for $R$ a commutative ring the diagonal of any product in $\mathcal{U}^R_{\mathcal{M}}$ is readily seen to be independent of the ordering of the terms.
\end{proof}
\noindent Next, we want to guarantee that the idempotents we construct are orthogonal. To this end we have the following transformation, inspired by the Gram-Schmidt procedure for orthogonalising a basis of a linear space with respect to an inner product. 
\begin{prop}\label{qq}
Let $E_1, \dots, E_m$ be mutually orthogonal idempotents in $\mathcal{U}_{\mathcal{M}}^{\Z}$. That is, we have $E_iE_j = \delta_{i,j}E_i$ for all $i,j \in \{1, \dots, m\}$. Suppose furthermore that $T \in \mathcal{U}_{\mathcal{M}}^{\Z}$ is an element with only zeroes and ones on the diagonal, in such a way that $TE_i$ has a vanishing diagonal for all $i \in \{1, \dots, m\}$. In other words, we have $\mathcal{D}_T \cap \mathcal{D}_{E_i} = \emptyset$ for all $i \in \{1, \dots, m\}$. We define the element
\begin{equation}
Q := T - \sum_{i =1}^m TE_i - \sum_{i =1}^m E_iT +  \sum_{i,j = 1}^m E_iTE_j \, .
\end{equation}
Then $Q$ has the same diagonal as $T$ and furthermore satisfies $QE_i = E_iQ = 0$ for all $i \in \{1, \dots, m\}$.
\end{prop}

\begin{proof}
By assumption, it holds that $TE_i$, $E_iT$ and $E_iTE_j$ for $i,j \in \{1, \dots, m\}$ all have a vanishing diagonal.  From this we see that $Q$  and $T$ indeed have the same diagonal. Next, let $k \in \{1, \dots, m\}$ be given. A direct calculation shows
\begin{align}
QE_k&= TE_k - \sum_{i =1}^m TE_iE_k - \sum_{i =1}^m E_iTE_k +  \sum_{i,j = 1}^m E_iTE_jE_k  \\ \nonumber
&= TE_k - \sum_{i =1}^m T\delta_{i,k}E_k - \sum_{i =1}^m E_iTE_k +  \sum_{i,j = 1}^m E_iT\delta_{j,k}E_k  \\ \nonumber
&= TE_k -  TE_k - \sum_{i =1}^m E_iTE_k +  \sum_{i = 1}^m E_iTE_k  = 0 \,.
\end{align}
A similar calculation shows that $E_kQ = 0$. This finishes the proof. 
\end{proof}
\noindent Finally, it remains to guarantee completeness of the system of idempotents we will construct. To this end, we have the following result.

\begin{prop}\label{sumis1t}
Suppose $E_1, \dots, E_p$ are mutually orthogonal idempotents such that $X := E_1 + \dots + E_p$ has only ones on its diagonal, then we necessarily have $X = \Id$.
\end{prop}

\begin{proof}
As the $E_i$ are mutually orthogonal idempotents, we also have $X^2 = X$. In particular, we find
\begin{equation}\label{Sortrip}
(\Id-X)^2 = \Id - 2X + X^2 = \Id -2X +X = \Id-X\, .
\end{equation}
Moreover, the diagonal of $X$ consists of only ones, so that $\Id-X$ is nilpotent. In other words, we have $(\Id-X)^n = 0$ for some $n \geq 1$. By applying equation \eqref{Sortrip} repeatedly we find
\begin{align}
0 &= (\Id-X)^n = (\Id-X)^{n-1} = \dots = \Id - X\, ,
\end{align}
so that indeed $X = \Id$. Alternatively, from the fact that $X$ has a diagonal consisting of only ones, we may conclude that there exists a real matrix $Y$ such that $XY = YX = \Id$. It again follows that $X = X^2Y = XY = \Id$. This completes the proof.
\end{proof}

\subsection{Summary; an algorithm for $\Z\mathcal{M}$}\label{summary}
Finally, we have everything in order to present the algorithm for finding a complete system of primitive orthogonal idempotents $E_1, \dots, E_p$ for $\Z\mathcal{M}$. Analogous to Definition \ref{defip}, we define the maps
\begin{align}
&\mathcal{P}_{a,b}: \Z\mathcal{M} \rightarrow \Z\mathcal{M} \\ \nonumber
&\mathcal{P}_{a,b}(X) = e - (e-X^{a})^b \, ,
\end{align}
for $a,b \in \Z_{\geq 0}$. We assume the $\mathcal{R}$-trivial monoid $\mathcal{M}$ is given, together with a generating set $\mathcal{S} \subset \mathcal{M}$ (for example $\mathcal{S} = \mathcal{M}$). As before, we set $n := \#\mathcal{M}$. The algorithm then proceeds in 3 steps.\\
\begin{enumerate}
\item For every element $\sigma \in \mathcal{M}$, we construct the set 
\begin{equation}
\mathcal{L}_{\sigma}^{\mathcal{S}} :=  \{\tau \in \mathcal{S} \, \mid \, \sigma  \tau = \sigma\}\, .
\end{equation}
We will denote all the mutually distinct sets we obtain this way by $\mathcal{L}_{\sigma_1}^{\mathcal{S}}$ up to $\mathcal{L}_{\sigma_p}^{\mathcal{S}}$, $p \in \{1, \dots, n\}$. In other words, $\{\sigma_1, \dots, \sigma_p\}$ is a complete system of representatives for the loop-type relation. We also set
\begin{equation}
d_i := \#\{\sigma \in \mathcal{M} \, \mid \, \mathcal{L}_{\sigma}^{\mathcal{S}} = \mathcal{L}_{\sigma_i}^{\mathcal{S}} \} \, ,
\end{equation}
for all $i \in \{1, \dots, p\}$. That is, $d_i$ denotes the number of elements in $\mathcal{M}$ that have the same loop-type as $\sigma_i$.
\item Next, we calculate the elements 
\begin{equation}\label{erjee4}
T_{i} := \prod_{\tau \in \mathcal{L}_{\sigma_i}^{\mathcal{S}}} \tau \hspace{-0.2cm} \prod_{\kappa \in \mathcal{S} \setminus \mathcal{L}_{\sigma_i}^{\mathcal{S}}} (e - \kappa) \in \Z\mathcal{M}\, ,
\end{equation}
for all $i \in \{1, \dots, p\}$, analogous to the elements $T_{[\sigma_i]}$ of Lemma \ref{diagfix}.
\item Next, we set $E_1 := \mathcal{P}_{a_1,b_1}(T_1)$, for any choice of $a_1$ and $b_1$ such that $a_1 \geq n-d_1$ and $b_1 \geq d_1$. We then construct the other $E_i$ recursively. If $E_1$ up to $E_{m-1}$ are found, we set  
\begin{equation}
Q_{m} := T_{m} - \sum_{i =1}^{m-1} T_{m}E_i - \sum_{i =1}^{m-1} E_iT_{m} +  \sum_{i,j = 1}^{m-1} E_iT_{m}E_j \, .
\end{equation}
We then define $E_{m} := \mathcal{P}_{a_{m}, b_{m}}(Q_{m})$, for any choice of $a_{m}$ and $b_{m}$ such that $a_{m} \geq n-d_{m}$ and $b_{m} \geq d_{m}$.
\end{enumerate}
\begin{thr}
The element $E_1, \dots, E_p \in \Z\mathcal{M}$ as constructed above form a complete system of primitive orthogonal idempotents.
\end{thr}

\begin{proof}
We prove the theorem by showing that the analogous elements in $\mathcal{U}_{\mathcal{M}}^{\Z}$ under the isomorphism of Lemma \ref{key1} form a complete system of primitive orthogonal idempotents. In doing so, we will use the same notation for the elements in $\Z\mathcal{M}$ as for the corresponding elements in $\mathcal{U}_{\mathcal{M}}^{\Z}$. \\
\indent First, Lemma \ref{loopiss} tells us that two element $\sigma, \sigma'$ are loop-type equivalent, if and only if $\mathcal{L}_{\sigma}^{\mathcal{S}} = \mathcal{L}_{\sigma'}^{\mathcal{S}}$. \\
Next, it follows from Lemma \ref{diagfix} that the diagonal of $T_i = T_{[\sigma_i]}$ consists of only zeroes and ones, with only ones at the $\sigma$-diagonal entries for which $\sigma \sim \sigma_i$. In particular, $T_i$ has exactly $d_i$ ones on its diagonal. Moreover, for any $\sigma \in \mathcal{M}$ there is exactly one $i \in \{1, \dots, p\}$ such that $(T_i)_{\sigma, \sigma} = 1$. \\
By part 3 of Lemma \ref{idemz}, the element $E_1 = \mathcal{P}_{a_1, b_1}(T_1)$ is an idempotent with the same diagonal as $T_1$. Let us therefore assume that the first $m-1$ elements $E_1, \dots, E_{m-1}$ that we have found using the algorithm are  mutually orthogonal idempotents such that $E_i$ has the same diagonal as $T_i$ for all $i \in \{1, \dots, m-1\}$. It follows that $T_{m}E_i$ has vanishing diagonal for all $i \in \{1, \dots, m-1\}$. Therefore, Proposition \ref{qq} tells us that $Q_{m}$ has the same diagonal as $T_{m}$, and furthermore satisfies $Q_{m}E_i = E_iQ_{m} = 0$ for all $i \in \{1, \dots, m-1\}$. It follows from part 3 of Lemma \ref{idemz} that $E_{m} = \mathcal{P}_{a_{m}, b_{m}}(Q_{m})$ is an idempotent with the same diagonal as $Q_{m}$, and therefore as $T_{m}$. \\
\indent Now, we may assume that $d_{m} < n = \#\mathcal{M}$. Namely, if we had $d_i = n$ for some $i \in \{1, \dots, p\}$ then every element of $\mathcal{M}$ would have the same loop-type as $\sigma_i$. In that case we would have $p=1$ and the algorithm would stop after finding $E_1$. Under the assumption $d_{m} < n$, we have $a_{m} \geq n-d_{m} > 0$, from which it may be seen that $\mathcal{P}_{a_{m}, b_{m}}(0) = 0$. Therefore, there exists a polynomial with integer coefficients $\mathcal{Q}_{a_{m}, b_{m}}$ such that $\mathcal{P}_{a_{m}, b_m}(X) = X\mathcal{Q}_{a_{m}, b_m}(X) = \mathcal{Q}_{a_{m}, b_m}(X)X$. It follows that 
\begin{equation}
E_iE_{m} = E_i\mathcal{P}_{{a_{m}, b_m}}(Q_{m}) = E_iQ_{m}\mathcal{Q}_{{a_{m}, b_m}}(Q_{m}) = 0\, .
\end{equation}
for all $i = \{1, \dots, m-1\}$.
Likewise, we find $E_{m}E_i = 0$ for all $i = \{1, \dots, m-1\}$. We conclude that $E_1, \dots, E_p$ are mutually orthogonal idempotents with the same diagonals as $T_1, \dots, T_p$. It furthermore follows from lemma \ref{idemz}, part 2, that every $E_i$ is primitive.\\
\indent Finally, it remains to show that $E_1 + E_2 + \dots + E_p = \Id$. However, as every $E_i$ has the same diagonal as $T_i$, we see that the sum $E_1 + E_2 + \dots + E_p$ has a diagonal consisting of only ones. It now follows from Proposition \ref{sumis1t} that indeed $E_1 + E_2 + \dots + E_p = \Id$. This completes the proof.
\end{proof}

\section{The general case }\label{gennn}
Finally, we show how the complete system of primitive orthogonal idempotents for $\Z\mathcal{M}$ constructed in the previous section leads to one for $R\mathcal{M}$. Here, $R$ is any ring with $1 \not= 0$ and with a given complete system of primitive orthogonal idempotents $\epsilon_1, \dots, \epsilon_q$. The main observation is that any ring $R$ has a copy of either $\Z$ or $\Z/m\Z$ for $m \in \N$ in its center. More precisely, we have 
\begin{equation}\label{incluiee}
\Z/m\Z \cong \{ \dots -1 + -1, -1, 0, 1, 1+1, \dots \} \subset R\, ,
\end{equation}
where we also allow for $m=0$ as $\Z/0\Z \cong \Z$. Note that we exclude the case $m = 1$, as this would imply $1 = 0$ in $R$. Using the identification of equation \eqref{incluiee}, we get a ring homomorphism  $\theta$ from $\Z$ to $R$, given explicitly by 
\begin{equation}
\theta(z) = 1+1+ \dots + 1 \,(z \text{ times})  \text{ for } z \in \Z.
\end{equation}
The homomorphism $\theta$ then induces a ring homomorphism $\Theta$ from $\Z\mathcal{M}$ to $R\mathcal{M}$, given by 
\begin{equation}
\Theta\left(\sum_{\sigma \in \mathcal{M}}r_{\sigma} \cdot \sigma \right) = \sum_{\sigma \in \mathcal{M}}\theta({r}_{\sigma}) \cdot \sigma \end{equation}
for $r_{\sigma} \in \Z$. In turn, we get a ring homomorphism $\Theta'$ from $\mathcal{U}_{\mathcal{M}}^{\Z}$ to $\mathcal{U}_{\mathcal{M}}^{R}$, given by 
\begin{equation}
\Theta'\left(\sum_{\sigma \in \mathcal{M}}r_{\sigma} U_{\sigma} \right) =  \sum_{\sigma \in \mathcal{M}}\theta({{r}}_{\sigma}) U_{\sigma}\, ,
\end{equation}
where we use $U_{\sigma}$ to denote the matrix of Section \ref{realisisa} in both $\mathcal{U}_{\mathcal{M}}^{\Z}$ and $\mathcal{U}_{\mathcal{M}}^{R}$. Note that it also holds that $\Theta'(X)_{i,j} = \theta(X_{i,j})$ for $X \in \mathcal{U}_{\mathcal{M}}^{\Z}$ and for all $i,j \in \{1, \dots, n\}$. Moreover, under the identification of equation \eqref{incluiee} we may also write $\theta(z) = \overline{z}$ for the class of $z \in \Z$ in $\Z/m\Z$. Motivated by this, we will simply write $\overline{X} := \Theta'(X)$ to denote the element in $\mathcal{U}_{\mathcal{M}}^{R}$ corresponding to $X \in \mathcal{U}_{\mathcal{M}}^{\Z}$.\\
\indent Let $E_1, \dots, E_p$ be a complete system of primitive orthogonal idempotents for $\mathcal{U}_{\mathcal{M}}^{\Z}$ obtained by the construction of Section \ref{summary}. We construct the corresponding elements $\overline{E}_1, \dots, \overline{E}_p $ in $\mathcal{U}_{\mathcal{M}}^{R}$. Note that we still have $\overline{E}_i\overline{E}_j = \delta_{i,j}\overline{E}_i$ for all $i,j \in \{1, \dots, p\}$, as well as $\sum_{i=1}^p\overline{E}_i = \overline{\Id} = \Id$. It may however happen that some of the elements $\overline{E}_i$ are not primitive anymore. We also note that the diagonal of every matrix $\overline{E}_i$ consists of only zeroes and ones, with a non-zero number of ones. This in particular implies that $\overline{E}_i \not= 0$. Finally, we note that the entries of $\overline{E}_i$ are all in the center of $R$. Using these observations, we may prove the following result.

\begin{thr}\label{fromztor}
Let $E_1, \dots, E_p$ be a complete system of primitive orthogonal idempotents for $\mathcal{U}_{\mathcal{M}}^{\Z}$ obtained by the construction of Section \ref{summary}, and let $\epsilon_1, \dots, \epsilon_q$  be a complete system of primitive orthogonal idempotents for $R$. Then, a complete system of primitive orthogonal idempotents for $R\mathcal{M}$ is given by $(\epsilon_i\overline{E}_j)_{i, j=1}^{q,p}$.
\end{thr}

\begin{proof}
As the entries of every $\overline{E}_i$ are in the center of $R$, we have
\begin{equation}
\epsilon_i\overline{E}_j\epsilon_k\overline{E}_l = \epsilon_i\epsilon_k\overline{E}_j\overline{E}_l = \delta_{i,k}\epsilon_i\delta_{j,l}\overline{E}_j\ = \delta_{i,k}\delta_{j,l}\epsilon_i\overline{E}_j\, ,
\end{equation}
for all $i, k \in \{1, \dots, q\}$ and $j, l \in \{1, \dots, p\}$. This shows that the elements $\epsilon_i\overline{E}_j$ are mutually orthogonal idempotents. Note also that none of the matrices $\epsilon_i\overline{E}_j$ equals $0$, as at least one diagonal entry of $\epsilon_i\overline{E}_j$ equals $\epsilon_i \not= 0$. Next, we see that
\begin{equation}
\sum_{i=1}^q \sum_{j=1}^p\epsilon_i\overline{E}_j = \sum_{j=1}^p\left(\sum_{i=1}^q \epsilon_i\right)\overline{E}_j =  \sum_{j=1}^p 1\cdot \overline{E}_j = \Id\, .
\end{equation}
It remains to show that every $\epsilon_i\overline{E}_j$ is primitive. To this end, let $X, Y \in \mathcal{U}_{\mathcal{M}}^{R}$ be two mutually orthogonal idempotents such that $X + Y = \epsilon_i\overline{E}_j$. In particular, for every $\sigma \in \mathcal{M}$ we have $X_{\sigma, \sigma} + Y_{\sigma, \sigma} = (\epsilon_i\overline{E}_j)_{\sigma, \sigma}$, with $X_{\sigma, \sigma}$ and $Y_{\sigma, \sigma}$ mutually orthogonal idempotents in $R$. \\
Let us first assume that $\sigma$ and $\sigma_j$ have different loop-type, so that $(\epsilon_i\overline{E}_j)_{\sigma, \sigma}= 0$. Setting $r := X_{\sigma, \sigma}$ we get $Y_{\sigma, \sigma} = -r$, $r^2 = r$ and $-r^2 = X_{\sigma, \sigma}Y_{\sigma, \sigma} = 0$. We conclude that $r = X_{\sigma, \sigma} = Y_{\sigma, \sigma} = 0$.\\
If on the other hand we have $\sigma \sim \sigma_j$, then we find $X_{\sigma, \sigma} + Y_{\sigma, \sigma} = (\epsilon_i\overline{E}_j)_{\sigma, \sigma} = \epsilon_i$ for $X_{\sigma, \sigma}$ and $Y_{\sigma, \sigma}$ mutually orthogonal idempotents. However, $\epsilon_i$ was assumed to be primitive, so that we find $X_{\sigma, \sigma} = 0$ or $Y_{\sigma, \sigma} = 0$. Finally, recall that diagonal entries corresponding to loop-type equivalent elements are always the same, by Lemma \ref{diagisloop}. It follows that either $X$ or $Y$ has a fully vanishing diagonal. Let us assume this is the case for $X$. We find that $X^n = 0$, so that we have $X = X^2 = \dots = X^n = 0$. This shows that $\epsilon_i\overline{E}_j$ is indeed primitive, thereby proving the theorem. 
\end{proof}

\begin{remk}
Given a monoid $\mathcal{M}$, we may form the opposite monoid $\mathcal{M}^{op}$. This latter monoid has the same set of elements and unit as $\mathcal{M}$, but with multiplication $\circ^{op}$ given by $\sigma \circ^{op} \tau := \tau \circ \sigma$ for $\sigma, \tau \in \mathcal{M}^{op}$. Here, $\circ$ denotes multiplication in $\mathcal{M}$. It follows that a monoid $\mathcal{M}$ is $\mathcal{L}$-trivial, if and only if $\mathcal{M}^{op}$ is $\mathcal{R}$-trivial. In particular, we may use our results to obtain a complete system of primitive orthogonal idempotents for a finite $\mathcal{L}$-trivial monoid as well. \\
\indent To this end, let $\mathcal{M}$ be an $\mathcal{L}$-trivial monoid and assume $R$ is a ring with a given system of  idempotents $\epsilon_1, \dots, \epsilon_q$. Similar to the opposite monoid, we may construct the opposite rings $R^{op}$ and $(R\mathcal{M})^{op}$. One then easily verifies that $(R\mathcal{M})^{op} \cong R^{op}\mathcal{M}^{op}$. Moreover, it follows from the definition of a complete system of primitive orthogonal idempotents that the identification between $R$ and $R^{op}$ as sets respects such systems. Hence, the elements $\epsilon_1, \dots, \epsilon_q$ also form a system of idempotents for $R^{op}$.  As $\mathcal{M}^{op}$ is $\mathcal{R}$-trivial, we get a system of idempotents $(\epsilon_i\overline{E}_j)_{i, j=1}^{q,p}$ for $R^{op}\mathcal{M}^{op}$. As before, it follows that $(\epsilon_i\overline{E}_j)_{i, j=1}^{q,p}$ is a system of idempotents for $R\mathcal{M} \cong (R^{op}\mathcal{M}^{op})^{op}$ as well. $\hfill \triangle$
\end{remk}
\begin{remk}
Let $\mathcal{M}$ be a finite $\mathcal{R}$- or  $\mathcal{L}$-trivial monoid. As every field $k$ has only the idempotents $0$ and $1$, a complete system of primitive orthogonal idempotents for $k\mathcal{M}$ is simply given by $\{\overline{E}_1, \dots, \overline{E}_p \}$. Thus, such a system is in essence independent of the chosen field $k$. This result is in stark contrast with the case of finite groups. For example, the group ring $(\Z/3\Z) k$ decomposes into two irreducible representations for $k = \R$, but in three for  $k = \C$. Therefore, a complete system of primitive orthogonal idempotents for $(\Z/3\Z) k$ has two elements for $k = \R$, but three for $k = \C$.  $\hfill \triangle$
\end{remk}

\section{An example}\label{exammp}
We will now illustrate our algorithm with an example. Let $\mathcal{M} = \{\bf{1,2,3,4,5}\}$ be the monoid with five elements and multiplication table given by Figure \ref{multitab}. The unit of $\mathcal{M}$ is $\boldsymbol{1}$, and it may  be seen that $\mathcal{M}$ is generated by the set $\mathcal{S} := \{\bf{1,2,3}\}$. By looking at the entries in each row, one can verify that $\mathcal{M}$ is in fact $\mathcal{R}$-trivial.
\begin{figure}\label{multitab}
\centering
\begin{tabu}{ c|[1.2pt] c | c | c | c |c}
$\circ$ & \bf{1} & \bf{2} & \bf{3} & \bf{4} & \bf{5}\\ \tabucline[1.2pt]{-}
\bf{1} & \bf{1} & \bf{2} & \bf{3} & \bf{4} & \bf{5}\\ 
\hline
\bf{2} & \bf{2} & \bf{2} & \bf{2} & \bf{2} & \bf{2}\\ 
\hline
\bf{3} & \bf{3} & \bf{4} & \bf{5} & \bf{5} & \bf{5}\\ 
\hline
\bf{4} & \bf{4} & \bf{4} & \bf{4} & \bf{4} & \bf{4}\\ 
\hline
\bf{5} & \bf{5} & \bf{5} & \bf{5} & \bf{5} & \bf{5}\\ 
\end{tabu}
\caption{The multiplication table of an $\mathcal{R}$-trivial monoid $\mathcal{M} = \{\bf{1,2,3,4,5}\}$}
\label{multitab}
\end{figure}
We will go through the algorithm of Section \ref{summary} to determine a complete system of primitive orthogonal idempotents for $\Z\mathcal{M}$.
\begin{enumerate}
\item We first determine the sets $\mathcal{L}_{\sigma}^{\mathcal{S}} :=  \{\tau \in \mathcal{S} \, \mid \, \sigma  \tau = \sigma\}\,$ for all $\sigma \in \mathcal{M}$. From the multiplication table $\ref{multitab}$ we may see that
\begin{align}
\mathcal{L}^{\mathcal{S}}_{\boldsymbol{1}} = \mathcal{L}^{\mathcal{S}}_{\boldsymbol{3}}  = \{  \boldsymbol{1} \} \, , \qquad \mathcal{L}^{\mathcal{S}}_{\boldsymbol{2}} = \mathcal{L}^{\mathcal{S}}_{\boldsymbol{4}}  = \mathcal{L}^{\mathcal{S}}_{\boldsymbol{5}} = \{  \boldsymbol{1}, \boldsymbol{2}, \boldsymbol{3} \} = \mathcal{S}\, .
\end{align}
It follows that the loop-type relation is given by $\boldsymbol{1} \sim \boldsymbol{3}$ and $\boldsymbol{2} \sim \boldsymbol{4} \sim \boldsymbol{5}$. In particular, every element has the same loop-type as either $\boldsymbol{1}$ or $\boldsymbol{2}$. The number of elements with the same loop-type as $\boldsymbol{1}$ is given by $d_1 = 2$ and the number of elements with the same loop-type as $\boldsymbol{2}$ is given by $d_2 = 3$.
\item Next, we calculate the elements $T_1$ and $T_2$:
\begin{align}\label{erjee11}
T_{1} &= \prod_{\tau \in \mathcal{L}_{\boldsymbol{1}}^{\mathcal{S}}} {\tau} \hspace{-0.2cm} \prod_{\kappa \in \mathcal{S} \setminus \mathcal{L}_{\boldsymbol{1}}^{\mathcal{S}}} (\boldsymbol{1} - {\kappa}) = \boldsymbol{1}(\boldsymbol{1} - \boldsymbol{2})(\boldsymbol{1} - \boldsymbol{3}) = \boldsymbol{1} - \boldsymbol{3} - \boldsymbol{2} + \boldsymbol{2} = \boldsymbol{1} - \boldsymbol{3} \, ,\\ \nonumber
T_{2} &= \prod_{\tau \in {\mathcal{S}}} {\tau} = \boldsymbol{1}\boldsymbol{2}\boldsymbol{3} = \boldsymbol{2}\, .
\end{align}
\item After determining $T_1$ and $T_2$, we calculate $E_1 = \mathcal{P}_{a_1,b_1}(T_1) = \boldsymbol{1} - ({\boldsymbol{1}}-T_1^{a_1})^{b_1}$ for $a_1 = n-d_1 = 5-2 = 3$ and $b_1 = d_1 = 2$:
\begin{align}\label{erjee12}
&T_1^3 = (\boldsymbol{1} - \boldsymbol{3})^3 = \boldsymbol{1} - 3\cdot\boldsymbol{3} + 2\cdot\boldsymbol{5} \\ \nonumber
&\boldsymbol{1}- T_1^{3} = 3\cdot\boldsymbol{3} - 2\cdot\boldsymbol{5}  \\ \nonumber
&(\boldsymbol{1}- T_1^{3} )^{2} = (3\cdot\boldsymbol{3} - 2\cdot\boldsymbol{5})^2 = \boldsymbol{5} \\ \nonumber
&\boldsymbol{1} - (\boldsymbol{1}- T_1^{3} )^{2} = \boldsymbol{1} - \boldsymbol{5} \, .
\end{align}
Hence, we find $E_1 =  \boldsymbol{1} - \boldsymbol{5}$. One verifies that indeed $E_1^2 = E_1$. \\
Next, we calculate $Q_2$:
\begin{align}
Q_{2} &= T_{2} -  T_{2}E_1 - E_1T_{2} + E_1T_{2}E_1 \\ \nonumber
&= \boldsymbol{2} -  \boldsymbol{2}(\boldsymbol{1} - \boldsymbol{5}) - (\boldsymbol{1} - \boldsymbol{5})\boldsymbol{2} + (\boldsymbol{1} - \boldsymbol{5})\boldsymbol{2}(\boldsymbol{1} - \boldsymbol{5}) \\ \nonumber
&= \boldsymbol{2} - (\boldsymbol{2} - \boldsymbol{2}) - (\boldsymbol{2} - \boldsymbol{5}) + (\boldsymbol{1} - \boldsymbol{5})(\boldsymbol{2} - \boldsymbol{2}) \\ \nonumber
&= \boldsymbol{2} - 0 - (\boldsymbol{2} - \boldsymbol{5}) + 0  = \boldsymbol{5}\, .
\end{align}
Note that $Q_2$ already satisfies $Q_2^2 = Q_2$, and that we have $E_1 + Q_2 = \boldsymbol{1}$. Therefore, we have found a complete system of primitive orthogonal idempotents if we set $E_2 = Q_2 = \boldsymbol{5}$. To complete our algorithm, for any $a_2 \geq 1$ and $b_2 \geq 1$ we have
\begin{align}
 \mathcal{P}_{a_2,b_2}(Q_2) = \boldsymbol{1} - (\boldsymbol{1}-\boldsymbol{5}^{a_2})^{b_2} = \boldsymbol{1} - (\boldsymbol{1}-\boldsymbol{5})^{b_2} = \boldsymbol{1} - (\boldsymbol{1}-\boldsymbol{5}) = \boldsymbol{5}\, ,
\end{align}
so that we indeed find $E_2 = \boldsymbol{5}$. We have thus found a complete system of primitive orthogonal idempotents given by $E_1 =  \boldsymbol{1} - \boldsymbol{5}$ and $E_2 = \boldsymbol{5}$. $\hfill \triangle$
\end{enumerate}

\bibliography{networks}
\bibliographystyle{unsrt}

\end{document}